\documentclass[a4paper]{amsart}

\usepackage[utf8]{inputenc}

\usepackage[english]{babel}
\usepackage{csquotes}

\usepackage{amssymb}
\usepackage{mathrsfs}
\usepackage{hyperref}
\usepackage{xcolor}
\usepackage{eucal}

\usepackage{amssymb}
\usepackage{amsmath}
\usepackage{amsthm}
\usepackage{imakeidx}
\usepackage{breqn}

\usepackage{diagbox}

\usepackage{verbatim}

\usepackage{fancyhdr}

\usepackage{tikz-cd}

\usepackage{graphicx}

\usepackage{xcolor}

\usepackage[url=false, doi=false, backend=bibtex]{biblatex}

\theoremstyle{definition}

\newtheorem{mydef}{Definition}[section]

\newtheorem{myque}[mydef]{Question}
\newtheorem*{myack}{Acknowledgements}

\theoremstyle{remark}

\theoremstyle{plain}

\newtheorem{mycol}[mydef]{Corollary}
\newtheorem{mysen}[mydef]{Theorem}
\newtheorem{mythm}{Theorem}
\newtheorem{mylem}[mydef]{Lemma}
\newtheorem{mypro}[mydef]{Proposition}
\newtheorem{myfact}[mydef]{Fact}

\newtheorem*{myclaim}{Claim}

\numberwithin{mydef}{section}

\DeclareMathOperator{\cof}{cof}
\DeclareMathOperator{\dom}{dom}

\DeclareMathOperator{\ZFC}{ZFC}

\DeclareMathOperator{\otp}{otp}
\DeclareMathOperator{\cf}{cf}

\DeclareMathOperator{\Add}{Add}

\DeclareMathOperator{\AP}{AP}

\DeclareMathOperator{\Coll}{Coll}

\DeclareMathOperator{\DSS}{DSS}

\DeclareMathOperator{\MM}{MM}
\DeclareMathOperator{\IA}{IA}
\DeclareMathOperator{\cl}{cl}
\DeclareMathOperator{\On}{On}

\newcommand{\dM}{\mathbb{M}}
\newcommand{\dP}{\mathbb{P}}
\newcommand{\dQ}{\mathbb{Q}}

\newcommand{\dT}{\mathbb{T}}

\newcommand{\uhr}{\upharpoonright}

\title{Cascading Variants of Internal Approachability} %TODO

\author{Hannes Jakob} %TODO

 %TODO

\date{\today}

\begin{document}
	
	%For IMPAN again
	
%	\baselineskip=17pt
	 %TODO
	
	%End
	
	\begin{abstract}
		We construct models in which there are stationarily many structures that exhibit different variants of internal approachability at different levels. This answers a question of Foreman-Todorcevic. We also show that the approachability property at $\mu$ is consistent with having a distinction between variants of internal approachability for stationarily many $N\in[H(\mu^+)]^{\mu}$, answering a question of Levine.
	\end{abstract}
	
	\maketitle
	
	\section*{Introduction}
	
	The variants of internal approachability were introduced by Foreman and Todorcevic in \cite{ForemanTodorLowenheimSkolem}. We say that a set $N$ of size $\mu$ ($\mu$ regular uncountable) is
	\begin{enumerate}
		\item \emph{internally unbounded} if $[N]^{<\mu}\cap N$ is unbounded in $[N]^{<\mu}$,
		\item \emph{internally stationary} if $[N]^{<\mu}\cap N$ is stationary in $[N]^{<\mu}$,
		\item \emph{internally club} if $[N]^{<\mu}\cap N$ contains a club in $[N]^{<\mu}$,
		\item \emph{internally approachable} if there is a $\subseteq$-increasing and continuous sequence $(a_i)_{i\in\mu}$ of elements of $[N]^{<\mu}$ such that $\bigcup_{i\in\mu}a_i=N$ and $(a_i)_{i<j}\in N$ for every $j<\mu$.
	\end{enumerate}
	
	Clearly, these properties are ordered in ascending strength (at least for sets $N$ which are elementary submodels of some $H(\Theta)$). Previously, Krueger showed (see  \cite{KruegerIntApproachAndRefl, KruegerIntClubApproachLarge, KruegerApplicMSI}) that it is relatively consistent that any two properties are distinct for stationarily many $N\in[H(\Theta)]^{\mu}$. For the distinction between internal unboundedness and stationarity, this is only known for $\mu=\omega_1$ while the other properties can also be separated for larger $\mu$. More recently, the author and Levine (see \cite{JakobDisjointInterval, JakobLevineICIAInterval, LevineDisjointStatSeq}) extended Kruegers result to obtain stationarily many $N\in[H(\Theta)]^{\mu}$ which are internally stationary but not internally club or internally club but not internally approachable for all $\mu=\aleph_{n+1}$ ($n\in\omega$) and $\Theta>\mu$ simultaneously.
	
	In this paper, we will focus on a different direction of research pertaining to the following observation: Whenever $M\prec H(\Theta)$ has size $\mu$ and is internally unbounded (stationary; club; approachable), the same is true for $M\cap H(\Theta')$ whenever $\Theta'\in M$. This raises the natural question of whether the converse is true as well , which was first asked by Foreman and Todorcevic in \cite{ForemanTodorLowenheimSkolem} (see also \cite[Question 4.6]{ForemanSurvey}): ``Suppose $\kappa$ is regular, $N\prec H(\Theta)$ and $N\cap[N\cap\kappa]^{\aleph_0}$ is stationary. Must $N\cap[N\cap\kappa^+]^{\aleph_0}$ be stationary?''. We will answer their question in the negative by showing:
	
	\begin{mythm}\label{Thm1}
		Assume $\MM$ and $2^{\omega_2}=\omega_3$. There exist stationarily many $N\in[H(\omega_3)]^{\omega_1}$ such that $N\cap H(\omega_2)$ is internally approachable and $N$ is not internally stationary.
	\end{mythm}
	
	So in particular, assuming $\MM$ and $2^{\omega_2}=\omega_3$, there exist stationarily many $N\prec H(\omega_3)$ with size $\omega_1$ such that $N\cap[N\cap\omega_2]^{<\omega_1}$ is stationary but $N\cap [N\cap\omega_3]^{<\omega_1}$ is not stationary.
	
	Using related techniques, we also obtain a similar result at cardinals above $\omega_1$:
	
	\begin{mythm}\label{Thm2}
		Assume $\kappa$ is $\kappa^+$-supercompact and $\mu<\kappa$ is regular. There is a forcing extension where $\kappa=\mu^+$ and there are stationarily many $N\in[H(\mu^{++})]^{\mu}$ such that $N\cap H(\mu^+)$ is internally approachable and $N$ is internally stationary but not internally club.
	\end{mythm}
	
	The construction used to obtain Theorem \ref{Thm2} can be modified to obtain a different result which answers another open problem: Under the assumption $2^{\mu}=\mu^+$ (in which case $|H(\mu^+)|=\mu^+$), distinctions between some variants of internal approachability are equivalent to the existences of certain combinatorial objects: Assuming $2^{\mu}=\mu^+$, Krueger showed in \cite[Theorem 6.5]{KruegerApplicMSI} that there is a \emph{disjoint stationary sequence on $\mu^+$} if and only if there are stationarily many $N\in[H(\mu^+)]^{\mu}$ which are internally unbounded but not internally club and a folklore theorem (see e.g. \cite[Lemma 1]{CoxFAAppSSR}) states that the \emph{approachability property fails at $\mu$}, i.e. $\mu^+\notin I[\mu^+]$, if and only if there are stationarily many $N\in[H(\mu^+)]^{\mu}$ which are internally unbounded but not internally approachable. We will show that in both theorems the assumption $2^{\mu}=\mu^+$ cannot be removed, answering a question of Levine (see \cite[Question 1]{LevineDisjointStatSeq}):
	
	\begin{mythm}\label{Thm3}
		Assume $\kappa$ is $\kappa^+$-supercompact and $\mu<\kappa$ is regular. There is a forcing extension where $\kappa=2^{<\mu}=\mu^+$, $2^{\mu}=\mu^{++}$ and the following holds:
		\begin{enumerate}
			\item $\mu^+\in I[\mu^+]$ (so there does not exist a disjoint stationary sequence on $\mu^+$),
			\item there are stationarily many $N\in[H(\mu^+)]^{\mu}$ which are internally stationary but not internally club.
		\end{enumerate}
	\end{mythm}
	
	The paper is organized as follows: In the first section, we introduce known definitions and results that will be used throughout the paper. In the second section, we prove Theorem \ref{Thm1}. In the third section, we recall Mitchell's forcing and its basic properties. In the fourth section, we prove a theorem which allows us to obtain both Theorem \ref{Thm2} and Theorem \ref{Thm3} as corollaries. We close with a few open questions.
	
	\begin{myack}
		The author wants to thank the anonymous referee to an earlier version of this paper for their diligent reading and immensely helpful referee report leading to an improvement of the manuscript.
	\end{myack}
	
	\section{Preliminaries}
	
	We assume the reader is familiar with the basics of forcing and the study of large cardinals. Good introductory material can be found in \cite{JechSetTheory}, \cite{KunenSetTheory} and \cite{KanamoriHigherInfinite}.
	
	We define the following forcing notions: For any regular cardinal $\delta$, $\Add(\delta)$ consists of ${<}\,\delta$-sized partial functions from $\delta$ to $2$, ordered by $\supseteq$. $\Coll(\mu,\delta)$ consists of ${<}\,\mu$-sized partial functions from $\mu$ to $\delta$, ordered by $\supseteq$. $\Add(\delta)$ is ${<}\,\delta$-closed, $(2^{<\delta})^+$-cc.\ and adds a new subset of $\delta$, while $\Coll(\mu,\delta)$ is ${<}\,\mu$-closed, $(\delta^{<\mu})^+$-cc and adds a surjection from $\mu$ onto $\delta$.
	
	Shelah introduced the \emph{approachability ideal} in \cite{ShelahApproachability} in order to obtain results regarding the preservation of stationary subsets of $\delta\cap\cof({<}\,\mu)$ by ${<}\,\mu$-closed forcing notions.
	
	\begin{mydef}
		Let $\mu$ be a cardinal.
		\begin{enumerate}
			\item The \emph{approachability ideal on $\mu^+$}, denoted by $I[\mu^+]$, is defined as follows: $A\in I[\mu^+]$ if there exists a sequence $(a_{\alpha})_{\alpha<\mu^+}$ of elements of $[\mu^+]^{<\mu}$ and a club $C\subseteq\mu^+$ such that whenever $\gamma\in A\cap C$, there exists $E\subseteq\gamma$ with $\otp(E)=\cf(\gamma)$ such that $E$ is unbounded in $\gamma$ and $\{E\cap\alpha\;|\;\alpha<\gamma\}\subseteq\{a_{\alpha}\;|\;\alpha<\gamma\}$ (we say that \emph{$\gamma$ is approachable with respect to $(a_{\alpha})_{\alpha<\mu^+}$}).
			\item The \emph{approachability property at $\mu$}, denoted by $\AP_{\mu}$, is the statement that $\mu^+\in I[\mu^+]$.
		\end{enumerate}
	\end{mydef}
	
	As has been shown by Shelah in \cite{ShelahApproachability}, for any regular cardinal $\mu$, $I[\mu^+]$ is a ${<}\,\mu^+$-complete normal ideal. Furthermore, Shelah showed in \cite[Lemma 4.4]{ShelahApproachability} that $\mu^+\cap\cof(<\mu)\in I[\mu^+]$ for any regular cardinal $\mu$, so that $\AP_{\mu}$ is equivalent to stating that $\mu^+\cap\cof(\mu)\in I[\mu^+]$.
	
	The approachability ideal relates to the variants of internal approachability as follows (see e.g. \cite[Lemma 1]{CoxFAAppSSR}):
	
	\begin{mysen}[Folklore]\label{CoxThm}
		Let $\mu$ be a regular uncountable cardinal with $2^{\mu}=\mu^+$. Then $\AP_{\mu}$ fails if and only if there are stationarily many $N\in[H(\mu^+)]^{\mu}$ which are internally unbounded but not internally approachable.
	\end{mysen}
	
	Krueger introduced the notion of a \emph{disjoint stationary sequence} in \cite{KruegerApplicMSI}:
	
	\begin{mydef}
		Let $\mu$ be a regular uncountable cardinal. $\DSS(\mu^+)$ states that there exists a \emph{disjoint stationary sequence on $\mu^+$}, i.e. a sequence $(\mathcal{S}_{\alpha})_{\alpha\in S}$ such that the following holds:
		\begin{enumerate}
			\item $S\subseteq\mu^+\cap\cof(\mu)$ is stationary.
			\item For all $\alpha\in S$, $\mathcal{S}_{\alpha}$ is stationary in $[\alpha]^{<\mu}$ and for all $\alpha\neq\beta$, both in $S$, $\mathcal{S}_{\alpha}\cap\mathcal{S}_{\beta}=\emptyset$.
		\end{enumerate}
	\end{mydef}
	
	Krueger related the existence of a disjoint stationary sequence to the previous properties as follows (\cite[Corollary 3.7, Theorem 6.5]{KruegerApplicMSI}):
	
	\begin{mysen}[Krueger]\label{KruegerThm}
		Let $\mu$ be a regular cardinal.
		\begin{enumerate}
			\item If there exists a disjoint stationary sequence on $\mu^+$, $\AP_{\mu}$ fails.
			\item If $2^{\mu}=\mu^+$, then there exists a disjoint stationary sequence on $\mu^+$ if and only if there are stationarily many $N\in[H(\mu^+)]^{\mu}$ which are internally unbounded but not internally club.
		\end{enumerate}
	\end{mysen}
	
	For the preservation of stationarity in $[X]^{<\mu}$, we use the notion of \emph{weak internal approachability} (see \cite[Definition 2.1]{ForemanMagidorLCCounterexamples}):
	
	\begin{mydef}
		Let $\Theta$ be a cardinal and $N\prec H(\Theta)$, we say that $N$ is \emph{weakly internally approachable of length $\tau$} and write $N\in\IA(\tau)$ if there is a $\subseteq$-increasing and continuous sequence $(N_i)_{i<\tau}$ with $N=\bigcup_{i<\tau}N_i$ and $(N_i)_{i<j}\in N$ for every $j<\tau$.
	\end{mydef}
	
	This notion is enough to ensure stationary sets are preserved by sufficiently closed forcing notions (see the combination of Lemma 2.2 and Lemma 2.4 in \cite{KruegerApplicMSI}):
	
	\begin{myfact}[Krueger]\label{FactKrueger2}
		Assume $S\subseteq[H(\Theta)]^{<\mu}$ is stationary and $S\subseteq\IA(\tau)$ for some regular cardinals $\tau<\mu<\Theta$. If $\dP$ is ${<}\,\mu$-closed and $G$ is $\dP$-generic, $S$ is stationary in $[H(\Theta)^V]^{<\mu}$ in $V[G]$.
	\end{myfact}
	
	We will also make use of the following results due to Menas (see \cite{MenasStrongCSuperC} or \cite[Lemma 8.26, Theorem 8.27]{JechSetTheory} for a more modern source):
	
	\begin{myfact}[Menas]\label{FactMenas}
		Assume $X\subseteq Y\subseteq Z$ are sets and $\mu$ is a regular cardinal. For $A\subseteq [Y]^{<\mu}$, define $A\uhr X:=\{a\cap X\;|\; a\in A\}$ and $A\uparrow Z:=\{a\in[Z]^{<\mu}\;|\;a\cap Y\in A\}$.
		
		\begin{enumerate}
			\item If $C\subseteq[Y]^{<\mu}$ is a club, there is a function $F\colon[Y]^{<\omega}\to[Y]^{<\mu}$ such that
			$$\cl_F:=\{a\in[Y]^{<\mu}\;|\;\forall y\in[a]^{<\omega}\;F(y)\subseteq a\}\subseteq C$$
			and for any $F\colon[Y]^{<\omega}\to[Y]^{<\mu}$, $\cl_F$ is club in $[Y]^{<\mu}$.
			\item If $C\subseteq[Y]^{<\mu}$ is club in $[Y]^{<\mu}$, then $C\uhr X$ contains a club in $[X]^{<\mu}$ and $C\uparrow Z$ is club in $[Z]^{<\mu}$.
			\item If $S\subseteq[Y]^{<\mu}$ is stationary in $[Y]^{<\mu}$, then $S\uhr X$ is stationary in $[X]^{<\mu}$ and $S\uparrow Z$ is stationary in $[Z]^{<\mu}$.
		\end{enumerate}
	\end{myfact}
	
	In order to obtain models which are internally unbounded but not internally stationary (or internally stationary but not internally club), we use a construction due to Krueger: Refining a theorem of Gitik from \cite{GitikNonsplit}, Krueger showed the following (see \cite[Theorem 7.1]{KruegerApplicMSI}):
	
	\begin{myfact}[Krueger]\label{FactKrueger}
		Suppose $V\subseteq W$ are models of $\ZFC$ with the same ordinals and there is a real in $W\smallsetminus V$. Let $\mu$ be a regular uncountable cardinal in $W$ and let $X$ be a set in $V$ such that $(\mu^+)^W\subseteq X$. In $W$, let $\Theta\geq\mu^+$ be a regular cardinal such that $X\subseteq H(\Theta)^W$. Then in $W$ the set of all $N\in[H(\Theta)^W]^{<\mu}\cap\IA(\omega)$ with $N\cap X\notin V$ is stationary in $[H(\Theta)^W]^{<\mu}$.
	\end{myfact}
	
	By combining this with Fact \ref{FactKrueger2} and Fact \ref{FactMenas} we can see that e.g. in  $V[\Add(\omega)]$, $[\omega_2]^{<\omega_1}$ naturally splits into two disjoint stationary sets, namely $[\omega_2]^{<\omega_1}\cap V$ and $[\omega_2]^{<\omega_1}\smallsetminus V$, whose stationarity is preserved by countably closed forcing. So if we want to collapse $\omega_2$ to $\omega_1$ we can either use $\Coll(\omega_1,\omega_2)$ and preserve the stationarity of both sets or shoot a club through one of them. This is done using the following poset (see e.g.\cite[Definition 2.5]{KruegerIntClubApproachLarge}):
	
	\begin{mydef}
		Let $\mu$ be a regular uncountable cardinal and $X$ a set. Let $S\subseteq[X]^{<\mu}$ be stationary. $\dP(S)$ consists of functions $p\colon\alpha\to S$ such that $\alpha<\mu$ is a successor ordinal and $p$ is $\subseteq$-increasing and continuous, ordered by end-extension.
	\end{mydef}
	
	It is easy to see that $\dP(S)$ collapses $|X|$ to $|\mu|$. Under additional assumptions, $\dP(S)$ is ${<}\,\mu$-distributive and thus preserves all cardinals up to (and including) $\mu$.
	
	As is common for variants of Mitchell Forcing, we will use a projection analysis.
	
	\begin{mydef}
		Let $(\dP,\leq_{\dP})$ and $(\dQ,\leq_{\dQ})$ be posets. Then $\pi\colon\dP\to\dQ$ is a \emph{projection} if the following holds:
		\begin{enumerate}
			\item $\pi(1_{\dP})=1_{\dQ}$.
			\item If $p'\leq_{\dP} p$, $\pi(p')\leq_{\dQ}\pi(p)$.
			\item If $q\leq_{\dQ}\pi(p)$, there is $p'\leq_{\dP} p$ with $\pi(p')\leq_{\dQ} q$.
		\end{enumerate}
	\end{mydef}
	
	Projections are used for the following result (which is part of the folklore):
	
	\begin{myfact}\label{FactProjection}
		Assume $\dP$ and $\dQ$ are posets and $\pi\colon\dP\to\dQ$ is a projection. Let $H$ be $\dQ$-generic. In $V[H]$, let $\dP/H$ consist of all those $p\in\dP$ such that $\pi(p)\in H$. Then whenever $G$ is $\dP/H$-generic over $V[H]$, $G$ is $\dP$-generic over $V$ and $V[H][G]=V[G]$. Ergo, $V[G]$ is an extension of $V[H]$ using the poset $\dP/H$.
	\end{myfact}
	
	The predominant application of this fact is as follows: Suppose that $\pi\colon\dP\to\dQ$ is a projection and $\dP$ does not collapse certain cardinals. Then $\dQ$ also does not collapse the same cardinals, since being a cardinal is upwards absolute and any forcing extension by $\dQ$ is contained in a forcing extension by $\dP$.
	
	\section{An Application of Martin's Maximum}
	
	In this section, we show that $\MM$ implies that there are stationarily many structures which are ``cascadingly internally approachable''. Recall that $\MM$ (which was introduced in \cite{ForemanMagidorShelahMM}) states that whenever $\dP$ preserves stationary subsets of $\omega_1$ and $(D_{\alpha})_{\alpha<\omega_1}$ is a sequence of dense subsets of $\dP$, there is a filter $G\subseteq\dP$ which intersects every $D_{\alpha}$. We will use the following reformulation due to Woodin (see \cite{WoodinADFANS}, in the proof of Theorem 2.53): Recall that for a poset $\dP$ and a set $N$ we say that $G$ is \emph{$\dP$-generic over $N$} if for every $D\subseteq\dP$ which is an element of $N$ and dense in $\dP$ there is $p\in G\cap D\cap N$.
	
	\begin{myfact}[Woodin]\label{FactWoodin}
		Assume $\MM$. Let $\Theta$ be a regular cardinal. Whenever $\dP\in H(\Theta)$ preserves stationary subsets of $\omega_1$ there are stationarily many $N\in[H(\Theta)]^{<\omega_2}$ such that there exists a $\dP$-generic filter over $N$.
	\end{myfact}
	
	Define the following forcing notion:
	
	$$\dP:=\Coll(\omega_1,\omega_2)*\Add(\omega)*\dP([\check{\omega}_3]^{<\omega_1}\smallsetminus V[\Coll(\omega_1,\omega_2)])$$
	
	The idea behind this poset is as follows: $\Coll(\omega_1,\omega_2)$ collapses $\omega_2$ by adding a surjection from $\omega_1$ to $\omega_2$ such that all initial segments are in the ground model. However, the second poset collapses $\omega_3^V$ by adding a club subset of $[\omega_3^V]^{<\omega_1}$ which is disjoint from the ground model.
	
	To apply Fact \ref{FactWoodin}, we must show that $\dP$ preserves stationary subsets of $\omega_1$. For $\Coll(\omega_1,\omega_2)$, this holds because of the countable closure. $\MM$ implies $2^{\omega}=\omega_2$ and thus $\omega_2^{\omega}=\omega_2$. Ergo $|\Coll(\omega_1,\omega_2)|=\omega_2$ and the poset forces $\check{\omega}_3=\dot{\omega}_2$. Friedman and Krueger show the following in \cite[Proposition 4.3, Theorem 4.4]{FriedmanKruegerThinDisjoint}:
	
	\begin{myfact}[Friedman-Krueger]\label{FactFriedmanKrueger}
		$\Add(\omega)$ forces that $\dP([\omega_2]^{<\omega_1}\smallsetminus V)$ is ${<}\,\omega_1$-distributive. The two-step iteration $\Add(\omega)*\dP([\omega_2]^{<\omega_1}\smallsetminus V)$ preserves stationary subsets of $\omega_1$.
	\end{myfact}
	
	Since $\omega_3^V=\omega_2^{V[\Coll(\omega_1,\omega_2)]}$, $\dP$ indeed preserves stationary subsets of $\omega_1$.
	
	We will show that $\dP$ collapses $H(\omega_2)^V$ and $H(\omega_3)^V$ to have size $\omega_1$ in such a way that $H(\omega_2)^V$ becomes internally approachable and $H(\omega_3)^V$ does not become internally stationary.
	
	\begin{mylem}\label{PSeq}
		Assume $2^{\omega_2}=\omega_3$. $\dP$ forces the existence of $\subseteq$-increasing and continuous sequences $(a_i)_{i\in\omega_1}$ and $(b_i)_{i\in\omega_1}$ with the following properties:
		\begin{enumerate}
			\item $\bigcup_{i\in\omega_1}a_i=H(\omega_2)^V$ and for all $j<\omega_1$, $(a_i)_{i<j}\in[H(\omega_2)^V]^{<\omega_1}\cap H(\omega_2)^V$.
			\item $\bigcup_{i\in\omega_1}b_i=H(\omega_3)^V$ and for all $i<\omega_1$, $b_i\in[H(\omega_3)^V]^{<\omega_1}\smallsetminus H(\omega_3)^V$.
		\end{enumerate}
	\end{mylem}
	
	\begin{proof}
		Let $G$ be $\dP$-generic. Fix filters $H$, $I$ and $J$ such that $G=H*I*J$.
		
		We first prove (1): In $V$, $|H(\omega_2)|=|\omega_2^{\omega_1}|=\omega_2$ by $\MM$ (see e.g. \cite[Theorem 10]{ForemanMagidorShelahMM}). In $V[H]$, $|H(\omega_2)^V|=\omega_1$. Thus, let $F\colon\omega_1\to H(\omega_2)^V$ be a bijection in $V[H]$. Since $\Coll(\omega_1,\omega_2)$ is countably closed, $V^{<\omega_1}\cap V[H]\subseteq V$. In particular, for every $j<\omega_1$, $F[j]\in V$ and so $(F[i])_{i<j}\in V$. This clearly implies that the sequence $(a_i)_{i\in\omega_1}:=(F[i])_{i<\omega_1}$ has the desired properties.
		
		Now we prove (2): Let $F':=\bigcup J$. By a simple density argument, $F'$ is a $\subseteq$-increasing, continuous and cofinal function from $\omega_1$ into $[\omega_3^V]^{<\omega_1}\smallsetminus V[H]=[\omega_3^V]^{<\omega_1}\smallsetminus V$.
		
		In $V$, let $F''\colon\omega_3^V\to H(\omega_3)^V$ be a bijection. Then the sequence $(b_i)_{i\in\omega_1}:=(F''[F'(i)])_{i\in\omega_1}$ is again easily seen to be as required: For any $i\in\omega_1$, $F'(i)\notin V$. Since $F''\in V$ is a bijection, this implies $F''[F'(i)]\notin V$.
	\end{proof}
	
	This allows us to show Theorem \ref{Thm1}. Note that the assumptions of the theorem are consistent (relative to a supercompact cardinal) because $\MM$ can be forced assuming the existence of a supercompact cardinal (see \cite[Theorem 5]{ForemanMagidorShelahMM}) and is preserved by ${<}\,\omega_2$-directed closed forcing (see \cite[Theorem 4.3]{LarsonSepStatRefl})which can be used to force $2^{\omega_2}=\omega_3$.
	
	\begin{mysen}\label{MMThm}
		Assume $\MM$ and $2^{\omega_2}=\omega_3$. Then there exist stationarily many $N\in[H(\omega_3)]^{\omega_1}$ such that $N\cap H(\omega_2)$ is internally approachable and $N$ is not internally stationary.
	\end{mysen}
	
	\begin{proof}
		Fix names $(\dot{a}_i)_{i\in\omega_1}$ and $(\dot{b}_i)_{i\in\omega_1}$ for sequences exemplifying Lemma \ref{PSeq}. Let $C\subseteq[H(\omega_3)]^{\omega_1}$ be club and let $\Theta$ be large enough so that $\dP\in H(\Theta)$. Then $\{N\in[H(\Theta)]^{\omega_1}\;|\;N\cap H(\omega_3)\in C\wedge\omega_1\subseteq N\}$ is club in $[H(\Theta)]^{\omega_1}$ by Fact \ref{FactMenas}. By Fact \ref{FactWoodin} we can find $N\prec H(\Theta)$ with size $\omega_1$ such that the following holds:
		
		\begin{enumerate}
			\item $\dP,(\dot{a}_i)_{i\in\omega_1},(\dot{b}_i)_{i\in\omega_1}\in N$;
			\item there exists a $\dP$-generic filter $G$ over $N$;
			\item $\omega_1\subseteq N$;
			\item $N\cap H(\omega_3)\in C$.
		\end{enumerate}
		
		We will show that $N\cap H(\omega_2)$ is internally approachable and $N\cap H(\omega_3)$ is not internally stationary.
		
		To this end, for $i\in\omega_1$, let $a_i$ consist of all $x\in N\cap H(\omega_2)$ such that for some $p\in G\cap N$, $p\Vdash\check{x}\in\dot{a}_i$. Again for $i\in\omega_1$, let $b_i$ consist of all $x\in N\cap H(\omega_3)$ such that for some $p\in G\cap N$, $p\Vdash\check{x}\in\dot{b}_i$. We will show that the sequences $(a_i)_{i\in\omega_1}$ and $(b_i)_{i\in\omega_1}$ are as desired, starting with $(a_i)_{i\in\omega_1}$.
		
		\begin{myclaim}
			The following holds:
			\begin{enumerate}
				\item The sequence $(a_i)_{i\in\omega_1}$ is increasing and continuous;
				\item for every $i\in\omega_1$, $a_i\in[N\cap H(\omega_2)]^{<\omega_1}$;
				\item $\bigcup_{i\in\omega_1}a_i=N\cap H(\omega_2)$;
				\item for every $j\in\omega_1$, $(a_i)_{i<j}\in N\cap H(\omega_2)$.
			\end{enumerate}
		\end{myclaim}
		
		\begin{proof}
			We first show (1). Clearly, $a_i\subseteq N\cap H(\omega_2)$ by the definition. By Lemma \ref{PSeq}, $\dot{a}_i$ is forced to be countable, so we can let $\dot{f}$ be a name for a bijection from $\omega$ into $\dot{a}_i$. Whenever $x\in a_i$, there is a dense set $D_{x,i}$ of conditions forcing $\dot{f}(\check{n})=\check{x}$ for some $n\in\omega$. In particular, there is $p\in G\cap N$ forcing $\dot{f}(\check{n})=\check{x}$. Since $G$ is a filter, this $n$ is unique, which implies that $a_i$ is countable.
			
			Next up is (2). To see that the sequence is increasing, suppose that $i<j\in\omega_1$ and $x\in a_i$, witnessed by $p\in G\cap N$. Since $\Vdash\dot{a}_i\subseteq\dot{a}_j$ by Lemma \ref{PSeq}, $p$ also witnesses $x\in a_j$. To see that the sequence is continuous, suppose that $j\in\omega_1$ is a limit ordinal and $x\in a_j$, witnessed by $p\in G\cap D$. By Lemma \ref{PSeq}, the set $D_{x,j}$ consisting of all $r\in\dP$ such that $r\perp p$ or $r\leq p$ and $r\Vdash\check{x}\in\dot{a}_i$ for some $i<j$ is dense in $\dP$. By elementarity, $D_{x,j}\in N$ and so there is $r\in D_{x,j}\cap G\cap N$. Since $G$ is a filter, $r$ and $p$ are compatible, so we must have $r\leq p$ and $r\Vdash\check{x}\in\dot{a}_i$, i.e. $x\in a_i$.
			
			Now we show (3). To this end, let $x\in N\cap H(\omega_2)$. Let $D_x$ be the set of all $p\in\dP$ which force $\check{x}\in\dot{a}_i$ for some $i\in\omega_1$. By Lemma \ref{PSeq}, $D_x$ is dense in $\dP$ and it lies in $N$ by elementarity. Since $G$ is $\dP$-generic over $N$, there is $p\in D_x\cap G\cap N$. Thus $p$ witnesses that $x\in a_i$ for some $i\in\omega_1$.
			
			Lastly, we show (4). Let $j<\omega_1$. Let $D_j$ be the set of all $p\in\dP$ which force $(\dot{a}_i)_{i<j}=\check{x}$ for some $x\in H(\omega_2)$. Again be Lemma \ref{PSeq}, $D_j$ is dense in $\dP$. Furthermore, $D_j\in N$ by elementarity. Therefore, there is $p\in D_j\cap G\cap N$. The corresponding $x$ witnessing $p\in D_j$ is in $N$ as well by elementarity.
			
			We will show $x=(a_i)_{i<j}$. Since $p\Vdash\check{x}=(\dot{a}_i)_{i<j}$, $x$ is a sequence of length $j$ consisting of countable sets. We will show that $x(i)=a_i$ for every $i<j$. To this end, let $y\in x(i)$. Then $y\in N$, since $x(i)$ is countable and $p\Vdash\check{y}\in\dot{a}_i$. Thus, $y\in a_i$. On the other hand, let $y\in a_i$. Then there is $q\in G\cap N$ forcing $\check{y}\in\dot{a}_i$. Because $G$ is a filter and generic over $N$, there is $r\leq p,q$ in $G\cap N$. Then $r$ forces $\check{y}\in\check{x}(\check{i})$, which implies $y\in x(i)$.
		\end{proof}
		
		In summary, $(a_i)_{i\in\omega_1}$ witnesses that $N\cap H(\omega_2)$ is internally approachable. We now concern ourselves with $(b_i)_{i\in\omega_1}$.
		
		\begin{myclaim}
			The following holds:
			\begin{enumerate}
				\item $(b_i)_{i\in\omega_1}$ is $\subseteq$-increasing and continuous;
				\item for every $i\in\omega_1$, $b_i\in[N\cap H(\omega_3)]^{<\omega_1}$;
				\item $\bigcup_{i\in\omega_1}b_i=N\cap H(\omega_3)$;
				\item for every $i<\omega_1$, $b_i\notin N\cap H(\omega_3)$.
			\end{enumerate}
		\end{myclaim}
		
		\begin{proof}
			Points (1) through (3) follow using the same arguments as before. So we only need to show that (4) holds.
			
			So we show that $\{b_i\;|\;i\in\omega_1\}$ is disjoint from $N\cap H(\omega_3)$. To this end, let $i\in\omega_1$ and let $x\in N\cap H(\omega_3)$. Let $D_{x,i}$ be the set of all $p\in\dP$ such that for some $y$, $p\Vdash \check{y}\in\check{x}\triangle\dot{b}_i$. Then $D_{x,i}\in N$ by elementarity and $D_{x,i}$ is dense in $\dP$: Let $r\in\dP$ be arbitrary. By Lemma \ref{PSeq}, $\Vdash\dot{b}_i\neq\check{x}$. Thus $r$ forces that there is an element of $V$ which is either in $\dot{b}_i$ and not in $\check{x}$ or vice versa. Therefore we can find $p\leq r$ and $y$ such that $p\Vdash\check{y}\in\check{x}\triangle\dot{b}_i$.
			
			Ergo there is $p\in G\cap D_{x,i}\cap N$. The corresponding $y$ is in $N$ as well by elementarity. Then either $p\Vdash\check{y}\in\check{x}\smallsetminus\dot{b}_i$, in which case $y\in x\smallsetminus b_i$, as no element of $G$ can force $\check{y}\in\dot{b}_i$ or $p\Vdash\check{y}\in\dot{b}_i\smallsetminus\check{x}$, in which case $y\in b_i\smallsetminus x$. In summary, $b_i\neq x$. As $x$ was arbitrary, $b_i\notin N\cap H(\omega_3)$.
		\end{proof}
		
		In particular, $\{b_i\;|\;i\in\omega_1\}$ is club in $[N\cap H(\omega_3)]^{<\omega_1}$ and disjoint from $N\cap H(\omega_3)$, so $N\cap H(\omega_3)$ is not internally stationary.
	\end{proof}
	
	The last ingredient we need to answer Foreman's question is the following simple observation:
	
	\begin{mylem}\label{LemmaMapping}
		Let $M\prec H(\Theta)$ have size $\mu$ and let $X,Y\in M$ be such that $|X|=|Y|$. Then $M\cap[M\cap X]^{<\mu}$ is stationary in $[M\cap X]^{<\mu}$ if and only if $M\cap[M\cap Y]^{<\mu}$ is stationary in $[M\cap Y]^{<\mu}$ and $M\cap[M\cap X]^{<\mu}$ contains a club in $[M\cap X]^{<\mu}$ if and only if $M\cap[M\cap Y]^{<\mu}$ contains a club in $[M\cap Y]^{<\mu}$.
	\end{mylem}
	
	\begin{proof}
		By elementarity, $M$ contains a bijection $F$ between $X$ and $Y$. We just prove one statement (the others follow with the same arguments). Assume $M\cap[M\cap X]^{<\mu}$ is stationary in $[M\cap X]^{<\mu}$ and $c\subseteq[M\cap Y]^{<\mu}$ is a club. By elementarity, $F$ restricts to a bijection between $M\cap X$ and $M\cap Y$, so $\{F^{-1}[x]\;|\;x\in c\}$ is club in $[M\cap X]^{<\mu}$. Ergo there exists $x\in c$ with $F^{-1}[x]\in M$. But then $x=F[F^{-1}[x]]\in M\cap c$ as desired.
	\end{proof}
	
	\begin{mycol}
		Assume $\MM$ and $2^{\omega_2}=\omega_3$. Then there are stationarily many $N\in[H(\omega_3)]^{\omega_1}$ such that $N\cap[N\cap\omega_2]^{<\omega_1}$ is stationary in $[N\cap\omega_2]^{<\omega_1}$ and $N\cap[N\cap\omega_3]^{<\omega_1}$ is not stationary in $[N\cap\omega_3]^{<\omega_1}$.
	\end{mycol}
	
	\begin{proof}
		Let $C\subseteq[H(\omega_3)]^{\omega_1}$ be club. Using Fact \ref{FactMenas} and Theorem \ref{MMThm}, find $M\prec H(\Theta)$ for $\Theta$ large enough with $M\cap H(\omega_3)\in C$ such that $M\cap H(\omega_2)$ is internally approachable and $M\cap H(\omega_3)$ is not internally stationary. In particular, $M\cap H(\omega_2)$ is internally stationary.
		
		Since $M\cap[M\cap H(\Theta)]^{<\omega_1}=(M\cap H(\Theta))\cap[M\cap H(\Theta)]^{<\omega_1}$, it follows easily that $M\cap[M\cap H(\omega_2)]^{<\omega_1}$ is stationary in $[M\cap H(\omega_2)]^{<\omega_1}$ and $M\cap[M\cap H(\omega_3)]^{<\omega_1}$ is not stationary in $[M\cap H(\omega_2)]^{<\omega_1}$. So Lemma \ref{LemmaMapping} implies the desired result since $|H(\omega_2)|=\omega_2$ and $|H(\omega_3)|=\omega_3$.
	\end{proof}
	
	\section{Mitchell Forcing}
	
	We now focus on obtaining similar results at cardinals above $\omega_1$. For this purpose, we will use Mitchell forcing. The original Mitchell forcing was discovered by William Mitchell (see \cite{MitchellTreeProp}) and used to show that, consistently, $\omega_2$ can have the tree property. In more recent work, Krueger applied similar techniques (see e.g. \cite{KruegerMitchellStyle}) toward the separation of variants of internal approachability.
	
	Mitchell's forcing is a skillful combination of Cohen forcing and collapses. When constructing variants of Mitchell forcing, the following choice must be made: If collapses occur immediately in limit steps, inaccessible cardinals are collapsed in a ``fresh way'' -- i.e. every initial segment of the added surjection is in the relative ground model -- while, if collapses occur only after adding a Cohen real, the collapses of the inaccessible cardinals have the approximation property -- i.e. any added surjection necessarily has an initial segment which lies outside of the relative ground model. This implies that the former version forces $\AP_{\omega_1}$, while the latter version forces $\neg\AP_{\omega_1}$ (see e.g. \cite{CummingsFriedmanMagidorRinotSinapovaEightfold}).
	
	In this paper, we will use the former version. However, we will make sure that, while the inaccessible cardinal $\nu$ is collapsed in a fresh way, its successor $\nu^+$ is collapsed by a forcing with a suitable approximation property. This ensures that, later on, we obtain models such that $M\cap H(\omega_2)$ is internally approachable and $M\cap H(\omega_3)$ is not internally club.
	
	\begin{mydef}
		Let $\mu$ be regular uncountable and $\xi$ an ordinal. $\dM(\mu,\xi)$ consists of pairs $(p,q)$ such that the following holds:
		\begin{enumerate}
			\item $p\in\Add(\omega,\xi)$.
			\item $q$ is a ${<}\,\mu$-sized partial function on $\xi$ such that for all $\alpha\in\dom(q)$, $q(\alpha)$ is an $\Add(\omega,\alpha)$-name for an element of $\dot{\Coll}(\check{\mu},\check{\alpha})$.
		\end{enumerate}
		We let $(p',q')\leq(p,q)$ if
		\begin{enumerate}
			\item $p'\leq p$ in $\Add(\omega,\xi)$,
			\item $\dom(q')\supseteq\dom(q)$ and for all $\alpha\in\dom(q)$,
			$$p'\uhr\alpha\Vdash q'(\alpha)\leq q(\alpha)$$
		\end{enumerate}
	\end{mydef}
	
	For the rest of the paper, we fix a regular uncountable cardinal $\mu$. As is common when working with variants of Mitchell forcing, we explicitely define the quotient forcings of different instances:
	
	\begin{mydef}
		Let $\xi<\kappa$ be ordinals and let $G$ be an $\dM(\mu,\xi)$-generic filter. In $V[G]$, let $\dM(G,\mu,\kappa\smallsetminus\xi)$ consist of pairs $(p,q)$ such that the following holds:
		\begin{enumerate}
			\item $p\in\Add(\omega,\kappa\smallsetminus\xi)$.
			\item $q$ is a ${<}\,\mu$-sized partial function on $\kappa\smallsetminus\xi$ such that for all $\alpha\in\dom(q)$, $q(\alpha)$ is an $\Add(\omega,\alpha\smallsetminus\xi)$-name for an element of $\dot{\Coll}(\check{\mu},\check{\alpha})$.
		\end{enumerate}
		We let $(p',q')\leq(p,q)$ if
		\begin{enumerate}
			\item $p'\leq p$ in $\Add(\omega,\kappa\smallsetminus\xi)$,
			\item $\dom(q')\supseteq\dom(q)$ and for all $\alpha\in\dom(q)$,
			$$p'\uhr(\alpha\smallsetminus\xi)\Vdash q'(\alpha)\leq q(\alpha)$$
		\end{enumerate}
	\end{mydef}
	
	We also explicitely define a term ordering which enables us to obtain a projection analysis of our forcings:
	
	\begin{mydef}
		Let $\xi<\kappa$ be a ordinals.
		\begin{enumerate}
			\item Let $\dT(\mu,\kappa)$ consist of all $(p,q)\in\dM(\mu,\kappa)$ such that $p=\emptyset$, ordered as a suborder of $\dM(\mu,\kappa)$.
			\item Let $G$ be $\dM(\mu,\xi)$-generic. In $V[G]$, let $\dT(G,\mu,\kappa\smallsetminus\xi)$ consist of those $(p,q)\in\dM(\mu,\kappa)$ such that $p=\emptyset$, ordered as a suborder of $\dM(G,\mu,\kappa\smallsetminus\xi)$.
		\end{enumerate}
	\end{mydef}
	
	The following properties are standard, so we omit the arguments. For proofs, see e.g. \cite[Section 2]{LevineDisjointStatSeq}.
	
	\begin{mylem}\label{MOmnibus}
		Let $\xi<\kappa$ be ordinals such that $\kappa$ is an inaccessible cardinal.
		\begin{enumerate}
			\item $\dM(\mu,\kappa)$ is $\kappa$-Knaster.
			\item There is a projection from $\Add(\omega,\kappa)\times\dT(\mu,\kappa)$ onto $\dM(\mu,\kappa)$.
			\item If $G$ is $\dM(\mu,\xi)$-generic, in $V[G]$ there is a projection from $\Add(\omega,\kappa\smallsetminus\xi)\times\dT(G,\mu,\kappa\smallsetminus\xi)$ onto $\dM(G,\mu,\kappa\smallsetminus\xi)$.
			\item $\dT(\mu,\kappa)$ is ${<}\,\mu$-closed and $\dT(G,\mu,\kappa\smallsetminus\xi)$ is ${<}\,\mu$-closed in $V[G]$ whenever $G$ is $\dM(\mu,\xi)$-generic.
			\item For every $\delta\leq\kappa$, if $x\in [V[\Add(\omega,\delta)]]^{<\mu}\cap V[\dM(\mu,\delta)]$, then $x\in V[\Add(\omega,\delta)]$.
		\end{enumerate}
	\end{mylem}
	
	It follows that $\dM(\mu,\kappa)$ preserves all cardinals up to (and including) $\mu$. Furthermore, it is easy to see that $\dM(\mu,\kappa)$ collapses any cardinal between $\mu$ and $\kappa$, so it forces $\kappa=\mu^+$. In particular, we have the following (by a nice name argument and since $|\dM(\mu,\kappa)|=\kappa$):
	
	\begin{mycol}\label{CardCol}
		Let $\kappa$ be an inaccessible cardinal. 
		$\dM(\mu,\kappa)$ forces $2^{\omega}=2^{\mu}=\mu^+=\kappa$ and $2^{\kappa}=(2^{\kappa})^V$.
	\end{mycol}
	
	For our results, we also require certain properties for the quotient of $\dM(\mu,\kappa)$ and an initial segment $\dM(\mu,\xi)$. These results are provided by Lemma \ref{MOmnibus} together with the fact that, whenever $G$ is $\dM(\mu,\xi)$-generic, the poset $\dM(G,\mu,\kappa\smallsetminus\xi)$ acts as the quotient forcing (see e.g. \cite[Lemma 2.1.10]{FontanellaPhD}):
	
	\begin{mylem}\label{MDecomp}
		Let $\xi<\kappa$ be ordinals above $\mu$. There is a dense embedding from $\dM(\mu,\kappa)$ into $\dM(\mu,\xi)*\dM(\Gamma,\mu,\kappa\smallsetminus\xi)$.
	\end{mylem}
	
	\begin{proof}
		We first define the map $\iota$ and then show that it is indeed a dense embedding.
		
		Let $(p,q)\in\dM(\mu,\kappa)$. Let $\dot{q}_1$ be the following $\dM(\mu,\xi)$-name: It is forced that $\dom(\dot{q}_1)=\dom(q)\smallsetminus\xi$. For every $\alpha\in\dom(q)\smallsetminus\xi$, $q(\alpha)$ is an $\Add(\omega,\alpha)$-name for an element of $\dot{\Coll}(\check{\mu},\check{\alpha})$. Since $\Add(\omega,\alpha)$ is equivalent to $\Add(\omega,\xi)*\Add(\omega,\alpha\smallsetminus\xi)$, we may view $q(\alpha)$ as an $\Add(\omega,\xi)$-name for an $\Add(\omega,\alpha\smallsetminus\xi)$-name for an element of $\dot{\Coll}(\check{\mu},\check{\alpha})$. We let $\dot{q}_1(\check{\alpha})$ be this $\Add(\omega,\xi)$-name and note that $\dot{q}_1(\check{\alpha})$ is also an $\dM(\mu,\xi)$-name.
		
		Now we define
		$$\iota((p,q)):=((p\uhr\xi,q\uhr\xi),(p\uhr[\xi,\kappa),\dot{q}_1))$$
		It is evident that the image of $\iota$ is actually contained in $\dM(\mu,\xi)*\dM(\Gamma,\mu,\kappa\smallsetminus\xi)$ and that $\iota$ preserves the forcing order.
		
		So let us show that the image of $\iota$ is dense in $\dM(\mu,\xi)*\dM(\Gamma,\mu,\kappa\smallsetminus\xi)$. To this end, let $((p_0,q_0),\dot{x})\in\dM(\mu,\xi)*\dM(\Gamma,\mu,\kappa\smallsetminus\xi)$.
		
		Let $G$ be $\dM(\mu,\xi)$-generic containing $(p_0,q_0)$ and work in $V[G]$. Then $\dot{x}^G=(p_1,q_1)$, where $p_1\in\Add(\omega,\kappa\smallsetminus\xi)$ and $q_1$ is a ${<}\,\mu$-sized partial function on $\kappa\smallsetminus\xi$ such that for every $\alpha\in q_1$, $q_1(\alpha)$ is an $\Add(\omega,\alpha\smallsetminus\xi)$-name for an element in $\dot{\Coll}(\check{\mu},\check{\alpha})$. By the ccc of $\Add(\omega,\alpha\smallsetminus\xi)$, we may assume that for every $\alpha\in\dom(q_1)$, the name $q_1(\alpha)$ has size ${<}\,\mu$. In particular, be Lemma \ref{MOmnibus} (5), $q_1(\alpha)\in V[G_{\Add}]$, where $G_{\Add}$ is the $\Add(\omega,\xi)$-generic filter induced by $G$. Since $|q_1|<\mu$, $q_1\in V[G_{\Add}]$. Again by the ccc of $\Add(\omega,\xi)$, there is $y\in V$ such that $\dom(q_1)\subseteq y$. By extending $q_1$ if necessary, we may assume that $\dom(q_1)=y$.
		
		Thus, we can find a condition $(p_0',q_0')\in G$ extending $(p_0,q_0)$ which forces all of the preceding statements. This means that $(p_0',q_0')\Vdash\dot{x}=(\check{p}_1,\dot{q}_1)$, where $p_1\in\Add(\omega,\kappa\smallsetminus\xi)$ and $\dot{q}_1$ is an $\Add(\omega,\xi)$-name such that $(p_0',q_0')\Vdash\dom(\dot{q}_1)=\check{y}$.
		
		We are now ready to define a preimage. Let $p:=p_0'\cup p_1$ and let $q$ be a function with domain $\dom(q_0')\cup y$ defined as follows: For $\alpha\in\dom(q_0')$, $q(\alpha)=q_0'(\alpha)$. For $\alpha\in y$, $\dot{q}_1(\alpha)$ is an $\Add(\omega,\xi)$-name for an $\Add(\omega,\alpha\smallsetminus\xi)$-name for an element of $\dot{\Coll}(\check{\mu},\check{\alpha})$. We let $q(\alpha)$ be the $\Add(\omega,\alpha)$-name corresponding to $\dot{q}_1(\alpha)$.
		
		It now follows easily that $(p,q)\in\dM(\mu,\kappa)$ and $\iota(p,q)=((p_0',q_0'),(\check{p}_1,\dot{q}_1))$, which implies that $\iota(p,q)\leq((p_0,q_0),\dot{x})$ as desired.
	\end{proof}
	
	We record the following easy statement for later:
	
	\begin{mylem}\label{PartialEval}
		Assume $\xi$ is a limit ordinal. Then $\dM(\mu,\xi+1)$ is forcing equivalent to $\dM(\mu,\xi)*(\dot{\Coll}(\check{\mu},\check{\xi})\times\Add(\omega))$.
	\end{mylem}
	
	\begin{proof}
		Map $(p,q)\in\dM(\mu,\xi+1)$ to $((p,q\uhr\xi),(q(\xi),p(\xi))$. Since any condition in $\dot{\Coll}(\check{\mu},\check{\xi})^{\dM(\mu,\xi)}$ has size ${<}\,\mu$, it is actually in $V[\Add(\omega,\xi)]$ which makes it easy to see that the mapping is a dense embedding.
	\end{proof}
	
	Lastly, the following proposition highlights the reason why forcing with $\dM(\mu,\kappa)$ works as intended. For later applications, we are proving it in a stronger form, where we take the product of $\dM(\mu,\kappa)$ and some ${<}\,\mu$-closed poset which is $\nu$-Knaster for some $\nu<\kappa$.
	
	\begin{mypro}\label{MainProp}
		Suppose that $\nu<\kappa$ are inaccessible cardinals above $\mu$. Let $\dP$ be any ${<}\,\mu$-closed poset which is $\nu$-Knaster.
		
		Let $G\times H$ be an $\dM(\mu,\kappa)\times\dP$-generic filter and let, for $\xi<\kappa$ any ordinal, $G(\xi)$ be the $\dM(\mu,\xi)$-generic filter induced by $G$. Then the following holds:
		
		\begin{enumerate}
			\item In $V[G\times H]$, $\mu^+\in I[\mu^+]$.
			\item In $V[G\times H]$, the set $[(\nu^+)^V]^{<\mu}\cap V[G(\nu)]$ is stationary and co-stationary in $[(\nu^+)^V]^{<\mu}$.
		\end{enumerate}
	\end{mypro}
	
	\begin{proof}
		It is well-known (see e.g. \cite{CummingsFriedmanMagidorRinotSinapovaEightfold}) that in $V[G]$, $\mu^+\in I[\mu^+]$. By Lemma \ref{MOmnibus} and Easton's Lemma, again in $V[G]$, $\dP$ is ${<}\,\mu$-distributive and $\kappa=\mu^+$-cc. Therefore, forcing with $\dP$ over $V[G]$ does not change cofinalities. Ergo, any sequence witnessing $\mu^+\in I[\mu^+]$ in $V[G]$ still witnesses $\mu^+\in I[\mu^+]$ in $V[G\times H]$.
		
		For (2), define $\delta:=(\nu^+)^V$. Let us first note that, because of the ${<}\,\mu$-distributivity of $\dP$ in $V[G(\nu)]$, $[\delta]^{<\mu}\cap V[G(\nu)]=[\delta]^{<\mu}\cap V[G(\nu)\times H]$. We first show that the set is co-stationary.
		
		To this end, work in $V[G(\nu)\times H]$. We want to apply Fact \ref{FactKrueger} with $V[G(\nu)\times H]$ in lieu of $V$, $V[G(\nu+1)\times H]$ in lieu of $W$, $\mu$ in lieu of $\mu$ and $\delta$ in lieu of both $X$ and $\Theta$. To this end, note the following:
		\begin{enumerate}
			\item In $V[G(\nu+1)\times H]$, $\mu$ is a regular uncountable cardinal, since $\dP$ is ${<}\,\mu$-distributive in $V[G(\nu+1)]$ and $\dM(\mu,\nu+1)$ preserves $\mu$.
			\item By Lemma \ref{PartialEval}, $V[G(\nu+1)\times H]$ is an extension of $V[G(\nu)\times H]$ using the poset $(\Coll(\mu,\nu)\times\Add(\omega))^{V[G(\nu)]}$. In particular, $(\mu^+)^{V[G(\nu+1)\times H]}=(\nu^+)^V=\delta$.
			\item By the same argument, $V[G(\nu+1)\times H]\smallsetminus V[G(\nu)\times H]$ contains a real.
		\end{enumerate}
		
		Ergo, by Fact \ref{FactKrueger}, in $V[G(\nu+1)\times H]$, the set $S$ of all $N\in[H(\delta)^{V[G(\nu+1)\times H]}]^{<\mu}\cap\IA(\omega)$ with $N\cap\delta\notin V[G(\nu)\times H]$ is stationary in $[H(\delta)^{V[G(\nu+1)\times H]}]^{<\mu}$.
		
		By Lemma \ref{MDecomp}, $V[G\times H]$ is an extension of $V[G(\nu+1)\times H]$ using $\dM(G(\nu+1),\mu,\kappa\smallsetminus\nu+1)$, a poset which is a projection of $\Add(\omega,\kappa\smallsetminus\nu+1)\times\dT(G(\nu+1),\mu,\kappa\smallsetminus\nu+1)$. By Lemma \ref{MDecomp}, $\dT(G(\nu+1),\mu,\kappa\smallsetminus\nu+1)$ is ${<}\,\mu$-closed in $V[G(\nu+1)]$ and, by the ${<}\,\mu$-distributivity of $\dP$ in the same model, it is still ${<}\,\mu$-closed in $V[G(\nu+1)\times H]$.
		
		Ergo, by Fact \ref{FactKrueger2}, the stationarity of $S$ is preserved by forcing with $\dT(G(\nu+1),\mu,\kappa\smallsetminus\nu+1)$ over $V[G(\nu+1)\times H]$. The further forcing with $\Add(\omega,\kappa\smallsetminus\nu+1)$ also preserves the stationarity of $S$ by the ccc. Hence, by Fact \ref{FactProjection}, we can see that forcing with $\dM(G(\nu+1),\mu,\kappa\smallsetminus\nu+1)$ over $V[G(\nu+1)\times H]$ preserves the stationarity of $S$ and thus $S$ is stationary in $[H(\delta)^{V[G(\nu+1)\times H]}]^{<\mu}$ in $V[G\times H]$.
		
		Lastly, by Fact \ref{FactMenas}, this implies that the set $S':=\{N\cap\delta\;|\;N\in S\}$ is stationary in $[\delta]^{<\mu}$ in $V[G\times H]$. By the construction, $S'$ is disjoint from $V[G(\nu)\times H]$.
		
		The stationarity of $S$ follows in a very similar way: In $V[G(\nu)\times H]$, the set $S:=[H(\delta)^{V[G(\nu)\times H]}]^{<\mu}\cap\IA(\omega)$ is clearly stationary. Just like before, $V[G\times H]$ is an extension of $V[G(\nu)\times H]$ using a poset which is a projection of a product of a ${<}\,\mu$-closed poset and a ccc poset. Thus, $S$ is still stationary in $[H(\delta)^{V[G(\nu)\times H]}]^{<\mu}$ in $V[G\times H]$, which implies that $\{N\cap\delta\;|\;N\in S\}$ is stationary in $[\delta]^{<\mu}$. Since the latter set is contained in $V[G(\nu)\times H]$, we are done.
	\end{proof}
	
	\section{Different Levels of Internal Approachability}
	
	In this section, we prove a theorem that will allow us to obtain Theorem \ref{Thm2} and Theorem \ref{Thm3} as corollaries. The idea is to transfer the results of Proposition \ref{MainProp} to elementary submodels of $H(\Theta)$. In order to accomplish this, we need submodels of size ${<}\,\kappa$ such that $|M\cap\kappa^+|=(M\cap\kappa)^+$, since otherwise, the results of Proposition \ref{MainProp} would not apply. These models are provided by the following large cardinal property:
	
	\begin{mydef}
		Let $\kappa\leq\lambda$ be cardinals. $\Pr(\kappa,\lambda)$ states that for every $\Theta\geq\lambda$ there are stationarily many $N\in[H(\Theta)]^{<\kappa}$ such that:
		\begin{enumerate}
			\item $\nu:=N\cap\kappa$ is inaccessible,
			\item $[N\cap\lambda]^{<\nu}\subseteq N$,
			\item For every cardinal $\mu\in[\kappa^+,\lambda]\cap N$, $\otp(N\cap\mu)$ is a cardinal.
		\end{enumerate}
	\end{mydef}
	
	Variants of this property have been considered e.g. by Zeman in \cite{ZemanConStrength} or Krueger in \cite{KruegerDestroyStatSets}. For $\lambda=\kappa$, $\Pr(\kappa,\lambda)$ is equivalent to $\kappa$ being a Mahlo cardinal. For $\lambda>\kappa$ we have the following upper bound:
	
	\begin{mylem}
		Assume $\lambda>\kappa$ are cardinals such that $\kappa$ is $|H(\lambda)|$-supercompact. Then $\Pr(\kappa,\lambda)$ holds.
	\end{mylem}
	
	\begin{proof}
		Let $\Theta\geq\lambda$ and $C\subseteq[H(\Theta)]^{<\kappa}$ club. Then the set $D:=\{M\cap H(\lambda)\;|\;M\in C\}$ contains a club in $[H(\lambda)]^{<\kappa}$ by Fact \ref{FactMenas}.
		
		Let $j\colon V\to M$ be a $|H(\lambda)|$-supercompact embedding, i.e. $j(\kappa)>|H(\lambda)^V|$ and ${}^{|H(\lambda)^V|}M\subseteq M$. Ergo $j[H(\lambda)^V]\in M$ and is a member of $j(D)$. So there is $N\in j(C)$ with $N\cap j(H(\lambda)^V)=j[H(\lambda)^V]$. Then the following holds:
		\begin{enumerate}
			\item $N\cap j(\kappa)=\kappa$ is inaccessible in $M$;
			\item $[N\cap j(\lambda)]^{<\kappa}\subseteq N$: We have $N\cap j(\lambda)=j[\lambda]$. If $x\in[N\cap j(\lambda)]^{<\kappa}$, $x=j[y]$ for some $y\in[\lambda]^{<\kappa}$ (so $y\in H(\lambda)^V$). However, as $|y|<\kappa$, $x=j[y]=j(y)\in N$.
			\item For any cardinal $\mu\in[j(\kappa)^+,j(\lambda)]\cap N$, $\mu=j(\nu)$ for some $\nu\in H(\lambda)^V$, so $\nu$ is a cardinal by elementarity. It follows that $\otp(N\cap j(\nu))=\otp(j[\nu])=\nu$ which is a cardinal in $V$ and thus in $M$.
		\end{enumerate}
		
		So by elementarity there is $N\in C$ such that $\nu:=N\cap\kappa$ is inaccessible, $[N\cap\lambda]^{<\nu}\subseteq N$ and $\otp(N\cap\mu)$ is a cardinal for every $\mu\in[\kappa^+,\lambda]\cap N$.
	\end{proof}
	
	A better upper bound (however with considerably more work) is the $\lambda$-ineffability of $\kappa$.\footnote{For a sketch, see https://mathoverflow.net/a/466030/138274} We can now prove our multi-purpose Theorem:
	
	\begin{mysen}\label{DistOmnibus}
		Assume $\kappa$ is inaccessible, $\Pr(\kappa,\kappa^+)$ holds and $2^{\kappa}=\kappa^+$. Let $\mu<\kappa$ be regular and $\dP$ a poset satisfying the following:
		\begin{enumerate}
			\item $\dP$ is ${<}\,\mu$-closed;
			\item there is $\delta<\kappa$ such that $\dP$ is $\delta$-Knaster.
			\item $\dP\subseteq H(\kappa^+)$.
		\end{enumerate}
		After forcing with $\dM(\mu,\kappa)\times\dP$, the following holds:
		\begin{enumerate}
			\item $\mu^+\in I[\mu^+]$.
			\item For any sufficiently large $\Theta$ there are stationarily many $M\in[H(\Theta)]^{\mu}$ such that $\cf(M\cap\mu^+)=\mu$ and $M\cap[M\cap\mu^{++}]^{<\mu}$ is stationary but does not contain a club in $[M\cap\mu^{++}]^{<\mu}$.
		\end{enumerate}
	\end{mysen}
	
	\begin{proof}
		Denote $\dM:=\dM(\mu,\kappa)$. Let $G\times H$ be $\dM\times\dP$-generic. For any $\gamma<\kappa$, let $G(\gamma)$ be the $\dM(\mu,\gamma)$-generic filter induced by $G$.
		
		We note that, by Corollary \ref{CardCol}, $\kappa=\mu^+$ in $V[G]$. As in the proof of Proposition \ref{MainProp}, forcing with $\dP$ over $V[G]$ collapses no cardinals and thus $\kappa=\mu^+$ in $V[G\times H]$. Again by Proposition \ref{MainProp}, $\mu^+\in I[\mu^+]$ holds in $V[G\times H]$.
		
		So we only have to show the existence of stationarily many $N\in [H(\Theta)^{V[G\times H]}]^{\mu}$ such that $\cf(N\cap\mu^+)=\mu$ and $N\cap[N\cap\mu^{++}]^{<\mu}$ is stationary and co-stationary in $[N\cap\mu^{++}]^{<\mu}$.
		
		To this end, let $F\colon[H(\Theta)^{V[G\times H]}]^{<\omega}\to[H(\Theta)^{V[G\times H]}]^{<\mu}$ and let $\dot{F}$ be a name such that $\dot{F}^{G\times H}=F$. Let $\Theta'$ be so large that $\dot{F}\in H^V(\Theta')$ and apply $\Pr(\kappa,\kappa^+)$ to find $N\prec H^V(\Theta')$ such that:
		\begin{enumerate}
			\item $\dot{F},\dM(\mu,\kappa),\dP\in N$;
			\item $\nu:=N\cap\kappa$ is inaccessible;
			\item $[N\cap\kappa^+]^{<\nu}\subseteq N$;
			\item $\otp(N\cap\kappa^+)$ is a cardinal.
		\end{enumerate}
		It follows easily that $\otp(N\cap\kappa^+)=\nu^+$: By assumption, $\otp(N\cap\kappa^+)\geq\nu^+$. On the other hand, whenever $c$ is an initial segment of $N\cap\kappa^+$, $c=N\cap\alpha$ for some $\alpha\in N\cap\kappa^+$. Since $N$ contains a bijection between $\kappa$ and $\alpha$, $N\cap\alpha$ and $N\cap\kappa$ have the same size. Thus, every initial segment of $N\cap\kappa^+$ has ordertype ${<}\,\nu^+$, which implies $\otp(N\cap\kappa^+)\leq\nu^+$. Our desired model is $M:=N[G\times H]\cap H^{V[G\times H]}(\Theta)$. Note that $M$ is closed under $F$, since $\dot{F}^{G\times H}\in N[G\times H]$. The $\kappa$-cc of $\dM(\mu,\kappa)\times\dP$ implies:
		
		\begin{myclaim}
			$N[G\times H]\cap V=N$.
		\end{myclaim}
		
		\begin{proof}
			Let $\tau\in N$ be such that $\tau^G\in V$. In particular, there is a maximal antichain $A$ of conditions in $\dM(\mu,\kappa)\times\dP$ which force $\tau=\check{x}$ for some $x\in V$. By elementarity, we can assume that $A\in N$, which implies $A\subseteq N$, as $N\cap\kappa\in\kappa$ and $|A|\in N\cap\kappa$. Thus the unique element $(m,p)$ in $A\cap G\times H$ is a member of $N$ and, again by elementarity, so is the $x$ such that $(m,p)\Vdash\tau=\check{x}$. Since $\tau^{G\times H}=x$, we are done.
		\end{proof}
		
		Thus, in particular, $M\cap\mu^{++}=M\cap\kappa^+=N\cap\kappa^+$ and $M\cap\mu^+=\nu$, which has cofinality $\mu$ in $V[G\times H]$. The following claim will connect Proposition \ref{MainProp} and the desired ``approachability type of $M$'':
		
		\begin{myclaim}
			Let $x\in[N\cap\kappa^+]^{<\mu}$. Then $x\in N[G\times H]$ if and only if $x\in V[G(\nu)]$.
		\end{myclaim}
		
		\begin{proof}
			In one direction, assume that $x\in V[G(\nu)]$. Let $\tau$ be an $\dM(\mu,\nu)$-name for $x$. Since $\nu$ is inaccessible, $\dM(\mu,\nu)$ is $\nu$-cc and a subset of $N$. We may thus code $\tau$ as an element of $[N\cap\kappa^+]^{<\nu}$, which implies that $x=\tau^{G(\nu)}=\tau^{G\times H}\in N[G\times H]$.
			
			Now assume that $x\in N[G\times H]$. Since $\dP$ is ${<}\,\mu$-distributive in $V[G]$ and thus in $N[G]$ by elementarity, $x\in N[G]$. Let $\tau\in N$ be an $\dM(\mu,\kappa)$-name for $x$. Since $\dM(\mu,\kappa)$ is $\kappa$-cc, we can find such a name with the property that it has size ${<}\,\kappa$. However, as before, this implies that $\tau\subseteq N$. In particular, $\tau$ is actually an $\dM(\mu,\kappa)\cap N=\dM(\mu,\nu)$-name and thus $x=\tau^{G(\nu)}\in V[G(\nu)]$.
		\end{proof}
		
		Now we can finish the proof. By Proposition \ref{MainProp}, in $V[G\times H]$, the set $[(\nu^+)^V]^{<\mu}\cap V[G(\nu)]$ is stationary and co-stationary in $V[G\times H]$. Since $|N\cap\kappa^+|=\nu^+$ in $V$, this clearly implies that $[N\cap\kappa^+]^{<\mu}\cap V[G(\nu)]$ is stationary and co-stationary in $V[G\times H]$. By the previous claim, $[N\cap\kappa^+]^{<\mu}\cap V[G(\nu)]=[N\cap\kappa^+]^{<\mu}\cap N[G\times H]$, which is clearly equal to $[M\cap\kappa^+]^{<\mu}\cap M$. Thus, $M$ is exactly as desired.
	\end{proof}
	
	Now we obtain the two theorems as corollaries: The proof of Theorem \ref{Thm2} is a bit more involved because we also have to show that approachable points are connected to internally approachable structures (so we are showing one direction of Theorem \ref{CoxThm}).
	
	\begin{mysen}\label{Theorem2}
		Assume $\kappa$ is an inaccessible cardinal such that $\Pr(\kappa,\kappa^+)$ holds and $2^{\kappa}=\kappa^+$. Let $\mu<\kappa$ be regular. After forcing with $\dM(\mu,\kappa)$, the following holds:
		\begin{enumerate}
			\item $\mu^+\in I[\mu^+]$.
			\item There are stationarily many $N\in[H(\mu^{++})]^{\mu}$ such that $N\cap H(\mu^+)$ is internally approachable and $N$ is internally stationary but not internally club.
		\end{enumerate}
	\end{mysen}
	
	\begin{proof}
		We apply Theorem \ref{MOmnibus} with $\dP:=\{\emptyset\}$. Let $G\times H$ be $\dM(\mu,\kappa)\times\dP$-generic. In $V[G\times H]$, we have $\mu^+\in I[\mu^+]$, so fix a sequence $\vec{a}$ and a club $C\subseteq\mu^+$ witnessing this fact. In $V[G\times H]$, we still have $|H(\kappa^+)^{V[G\times H]}|=|2^{\kappa}|=\kappa^+$ because $|\dM(\mu,\kappa)|=\kappa$.
		
		Work in $V[G\times H]$. Let $D\subseteq[H(\mu^{++})]^{\mu}$ be any club. By Theorem \ref{DistOmnibus}, find $M\prec H(\Theta)$ for $\Theta$ sufficiently large such that $\vec{a}\in M$, $M\cap H(\mu^{++})\in D$, $M\cap\mu^+\in C$, $\cf(M\cap\mu^+)=\mu$ and $M\cap[M\cap\mu^{++}]^{<\mu}$ is stationary but does not contain a club in $[M\cap\mu^{++}]^{<\mu}$. By Lemma \ref{LemmaMapping}, since $|H(\mu^{++})|=\mu^{++}$, $M\cap[M\cap H(\mu^{++})]^{<\mu}$ is stationary but does not contain a club in $[M\cap H(\mu^{++})]^{<\mu}$. This shows easily that $M\cap H(\mu^{++})$ is internally stationary but not internally club.
		
		Since $M\cap\mu^+\in C$ and $\cf(M\cap\mu^+)=\mu$, there exists $A\subseteq M\cap\mu^+$ unbounded with ordertype $\mu$ such that any initial segment of $A$ is in $\{a_{\alpha}\;|\;\alpha\in M\cap\mu^+\}$. In particular, any initial segment of $A$ is in $M$.
		
		Let $f\colon\mu\to A$ be the increasing enumeration of $A$ and let $(f_{\gamma})_{\gamma\in[\mu,\mu^)}\in M$ be a sequence such that for each $\gamma\in[\mu,\mu^+)$, $f_{\gamma}\colon\mu\to\gamma$ is a bijection. Note that $f_{\gamma}\in M$ for $\gamma\in M$. Also let $F\colon \mu^+\to H(\mu^+)$ be a bijection in $M$. Now define the sequence $(b_i)_{i\in\mu}$ as follows: For $i\in\mu$, let
		$$b_i:=F\left[\bigcup_{\gamma\in f[i]}f_{\gamma}[i]\right]$$
		Note that $(b_i)_{i\in\mu}$ is $\subseteq$-increasing, continuous and consists of ${<}\,\mu$-sized sets. So we only need to show that the initial segments of $(b_i)_{i\in\mu}$ are in $M$ and that $\bigcup_{i\in\mu}b_i=M\cap H(\mu^+)$.
		
		For any $j<\mu$, $(b_i)_{i<j}$ is definable from an initial segment of $f$ which is itself definable from an initial segment of $A$ and parameters in $M$, so it is in $M$. The only thing left to show is that $\bigcup_{i<\mu}b_i=M\cap H(\mu^+)$. To this end, let $x\in M\cap H(\mu^+)$. By elementarity there is $\alpha\in M\cap\mu^+$ with $x=F(\alpha)$. Furthermore, there is $j\in\mu$ with $f(j)>\alpha$. Lastly, there is $i\in\mu$ with $\alpha=f_{f(j)}(i)$. Then $x\in b_{\max(i,j)+1}$.
	\end{proof}
	
	And we can obtain Theorem \ref{Thm3} with a less involved proof:
	
	\begin{mysen}
		Assume $\kappa$ is inaccessible, $\Pr(\kappa,\kappa^+)$ holds and $2^{\kappa}=\kappa^+$. Let $\mu<\kappa$ be regular. After forcing with $\dM(\mu,\kappa)\times\Add(\mu,\kappa^+)$, the following holds:
		\begin{enumerate}
			\item $\mu^+\in I[\mu^+]$.
			\item There does not exist a disjoint stationary sequence on $\mu^+$.
			\item There are stationarily many $N\in[H(\mu^+)]^{\mu}$ which are internally stationary but not internally club.
		\end{enumerate}
	\end{mysen}
	
	\begin{proof}
		Let $G\times H$ be $\dM(\mu,\kappa)\times\Add(\mu,\kappa^+)$-generic. In $V[G\times H]$, $\kappa=\mu^+$, $2^{\mu}=\mu^{++}$, and so $|H(\mu^+)^{V[G\times H]}|=\mu^{++}$.
		
		Work in $V[G\times H]$. By Theorem \ref{DistOmnibus} we have $\AP_{\mu}$ which implies by Theorem \ref{KruegerThm} that there does not exist a disjoint stationary sequence on $\mu^+$.
		
		Let $C\subseteq[H(\mu^+)]^{\mu}$ be any club. Let $\Theta$ be sufficiently large. Using standard arguments and applying Theorem \ref{DistOmnibus}, find $M\prec H(\Theta)$ such that $M\cap H(\mu^+)\in C$ and $M\cap[M\cap\mu^{++}]^{<\mu}$ is stationary but does not contain a club in $[M\cap\mu^{++}]^{<\mu}$. By Lemma \ref{LemmaMapping}, $M\cap[M\cap H(\mu^+)]^{<\mu}$ is stationary but does not contain a club in $[M\cap H(\mu^+)]^{<\mu}$. This easily shows that $M\cap H(\mu^+)$ is internally stationary but not internally club.
	\end{proof}
	
	\section{Closing Remarks and Open Questions}
	
	We have seen before that, in case $|H(\mu^+)|>\mu^+$, the notions of internal approachability do not characterize $\AP_{\mu}$ or the existence of a disjoint stationary sequence on $\mu^+$. This is due to the fact that, while $\AP_{\mu}$ and $\DSS(\mu^+)$ deal with subsets of ordinals, the notions of internal approachability deal with subsets of $H(\mu^+)$. In this section observe that, when slightly modifying the variants of internal approachability, we obtain a characterization of $\DSS(\mu^+)$ and $\AP_{\mu}$ without relying on the assumption $|H(\mu^+)|=\mu^+$.
	
	\begin{mydef}
		Let $M$ be a set with size $\mu$, where $\mu$ is a regular uncountable cardinal.
		\begin{enumerate}
			\item $M$ is \emph{$\On$-internally unbounded} if $M\cap[M\cap\On]^{<\mu}$ is unbounded in $[M\cap\On]^{<\mu}$,
			\item $M$ is \emph{$\On$-internally stationary} if $M\cap[M\cap\On]^{<\mu}$ is stationary in $[M\cap\On]^{<\mu}$,
			\item $M$ is \emph{$\On$-internally club} if $M\cap[M\cap\On]^{<\mu}$ contains a club in $[M\cap\On]^{<\mu}$,
			\item $M$ is \emph{$\On$-internally approachable} if there exists an increasing and continuous sequence $(x_i)_{i<\mu}$ of elements of $[M\cap\On]^{<\mu}$ such that $\bigcup_{i<\mu}x_i=M\cap\On$ and $(x_i)_{i<j}\in M$ for every $j<\mu$.
		\end{enumerate}
	\end{mydef}
	
	If $|H(\mu^+)|=\mu^+$ and $M\prec H(\Theta)$ for a sufficiently large $\Theta$, Lemma \ref{LemmaMapping} and its proof shows that $M\cap H(\mu^+)$ is internally unbounded (stationary; club; approachable) if and only if $M\cap H(\mu^+)$ is $\On$-internally unbounded (stationary; club; approachable). Our previous arguments show that this is in general not the case if $|H(\mu^+)|>\mu^+$.
	
	The proof of Krueger's result on the existence of a disjoint stationary sequence (see \cite[Theorem 6.5]{KruegerApplicMSI}) can be readily modified to show:
	
	\begin{mysen}
		Assume $\mu$ is a regular uncountable cardinal. Then the following are equivalent:
		\begin{enumerate}
			\item There exists a disjoint stationary sequence on $\mu^+$;
			\item there are stationarily many $N\in[H(\mu^+)]^{\mu}$ which are $\On$-internally unbounded but not $\On$-internally club.
		\end{enumerate}
	\end{mysen}
	
	Also, a straightforward modification of Cox' proof of the folklore result relating the approachability to the distinction between internal unboundedness and approachability shows the following:
	
	\begin{mysen}
		Assume $\mu$ is a regular uncountable cardinal. Consider the following statements:
		\begin{enumerate}
			\item $\mu^+\notin I[\mu^+]$;
			\item There are stationarily many $N\in[H(\mu^+)]^{\mu}$ which are $\On$-internally unbounded but not $\On$-internally approachable.
		\end{enumerate}
		Then $(2)$ implies $(1)$. Moreover, if $(\mu^+)^{<\mu}=\mu^+$, $(1)$ implies $(2)$.
	\end{mysen}
	
	\begin{proof}
		Let us give a sketch of the argument. Assume $\vec{a}=(a_{\alpha})_{\alpha<\mu^+}$ and $D\subseteq\mu^+$ witness $\mu^+\in I[\mu^+]$. Let $C\subseteq[H(\mu^+)]^{\mu}$ consist of all those $M\cap H(\mu^+)$, where $M\prec H(\Theta)$ ($\Theta$ large enough), $\vec{a}\in M$, $D\in M$ and $M\cap\mu^+\in\mu^+$. $C$ contains a club by Fact \ref{FactMenas}. Let $M\cap H(\mu^+)\in C$. If $M$ is $\On$-internally unbounded, $\cf(M\cap\mu^+)=\mu$, since otherwise $M\cap H(\mu^+)$ would need to contain a ${<}\,\mu$-sized superset of an unbounded subset of $M\cap\mu^+$ which is impossible. By elementarity, $M\cap\mu^+\cap D$ is unbounded in $M\cap\mu^+$, so $M\cap\mu^+\in D$. Ergo there is $A\subseteq M\cap\mu^+$ unbounded with ordertype $\mu$ such that every initial segment of $A$ is in $\{a_{\alpha}\;|\;\alpha\in M\cap\mu^+\}$ and thus in $M$. Let $f\colon\mu\to A$ be its increasing enumeration and for $\gamma\in[\mu,\mu^+)$, let $f_{\gamma}\colon\mu\to\gamma$ be a bijection (for $\gamma\in M\cap\mu^+$, we can assume $f_{\gamma}\in M$). Now define $(b_i)_{i\in\mu}$ as follows: For $i\in\mu$, let
		$$b_i:=\bigcup_{\gamma\in f[i]}f_{\gamma}[i]$$
		It follows as in the proof of Theorem \ref{Theorem2} that $(b_i)_{i\in\mu}$ witnesses the $\On$-internal approachability of $M$.
		
		Now assume $(\mu^+)^{<\mu}=\mu^+$ and $\mu^+\notin I[\mu^+]$. Let $\vec{a}=(a_{\alpha})_{\alpha<\mu^+}$ enumerate $[\mu^+]^{<\mu}$. By assumption (and the fact that $\mu^+\cap\cof(<\mu)\in I[\mu^+]$ by \cite[Lemma 4.4]{ShelahApproachability}) the set of all $\gamma\in\mu^+\cap\cof(\mu)$ which are not approachable with respect to $\vec{a}$ is stationary in $\mu^+$.
		
		Let $D\subseteq[H(\mu^+)]^{\mu}$ be club and find a continuous sequence $(M_i)_{i\in\mu^+}$ of elements of $D$ such that $(M_i)_{i\leq j}\in M_{j+1}$ for every $j\in\mu^+$. It follows that the set $\{i\in\mu^+\;|\;M_i\cap\mu^+=i\}$ is club in $\mu^+$ and so there must be $i\in\mu^+$ such that $i=M_i\cap\mu^+$, $\cf(M_i\cap\mu^+)=\mu$ and $M_i\cap\mu^+$ is not approachable with respect to $\vec{a}$. It follows that $M_i$ is $\On$-internally unbounded: If $x\subseteq M_i\cap\On=M_i\cap\mu^+$ has size ${<}\,\mu$, $x\subseteq M_j$ for some $j<i$ and $M_j\in M_i$. Assume toward a contradiction that $M\cap H(\mu^+)$ is $\On$-internally approachable, witnessed by $(x_i)_{i<\mu}$. Then the set $A:=\{\sup(x_i)\;|\;i<\mu\}$ is unbounded in $M\cap\mu^+$ and has ordertype $\mu$. Furthermore, any initial segment of $A$ is definable from some $(x_i)_{i<j}$ for $j<\mu$ and thus in $M$. By elementarity, it is equal to some $a_{\alpha}$ for $\alpha<M\cap\mu^+$. So $M\cap\mu^+$ is approachable with respect to $\vec{a}$, a contradiction.
	\end{proof}
	
	We close with a few open questions. We have shown in Theorem \ref{Thm1} that one direction of Theorem \ref{CoxThm} and Theorem \ref{KruegerThm} can fail without the assumed cardinal arithmetic. We do not know if the same is true for the other direction:
	
	\begin{myque}
		Is it consistent that $\AP_{\mu}$ fails and there are club many $N\in[H(\mu^+)]^{\mu}$ which are either not internally unbounded or internally approachable? Is it consistent that $\DSS(\mu^+)$ holds and the collection of all $N\in[H(\mu^+)]^{\mu}$ which are internally unbounded but not internally club is nonstationary?
	\end{myque}
	
	Our other question concerns the ``improved version'' of Theorem \ref{CoxThm}, where we do not know whether this cardinal arithmetic assumption can be relaxed:
	
	\begin{myque}
		Is it consistent that $\mu^+\notin I[\mu^+]$ and the collection of all $N\in[H(\mu^+)]^{\mu}$ which are $\On$-internally unbounded but not $\On$-internally approachable is nonstationary (here necessarily $(\mu^+)^{<\mu}\geq\mu^{++}$)?
	\end{myque}
	
	\printbibliography
	
\end{document}